\documentclass[preprint,12pt]{elsarticle}

\usepackage{amssymb}
\usepackage{amsmath}

\usepackage[utf8]{inputenc}
\usepackage{a4wide}
\usepackage{latexsym,amsfonts,amsmath,amssymb,mathrsfs,url,amsthm}
\usepackage{mathtools}
\usepackage{color,graphicx}
\usepackage{lipsum}
\usepackage{hyperref}
\usepackage{xcolor}
\usepackage{cleveref}
\usepackage{bm}
\usepackage{comment}

\newtheorem{theorem}{Theorem}[section]
\newtheorem{lemma}[theorem]{Lemma}

\newtheorem{algorithm}[theorem]{Algorithm}
\newtheorem{remark}[theorem]{Remark}

\newtheorem{corollary}[theorem]{Corollary}

\newcommand{\rd}{\, \mathrm{d}}

\newcommand{\bszero}{\boldsymbol{0}}
\newcommand{\bsh}{\boldsymbol{h}}
\newcommand{\bsg}{\boldsymbol{g}}
\newcommand{\bsk}{\boldsymbol{k}}
\newcommand{\bsl}{\boldsymbol{\ell}}

\newcommand{\bsx}{\boldsymbol{x}}

\newcommand{\bsz}{\boldsymbol{z}}
\newcommand{\bsgamma}{\boldsymbol{\gamma}}
\newcommand{\bsDelta}{\boldsymbol{\Delta}}

\newcommand{\NN}{\mathbb{N}}
\newcommand{\PP}{\mathbb{P}}
\newcommand{\RR}{\mathbb{R}}
\newcommand{\ZZ}{\mathbb{Z}}
\newcommand{\Acal}{\mathcal{A}}
\newcommand{\Hcal}{\mathcal{H}}

\newcommand{\tmod}[1]{{\;(\mathrm{mod}\; #1)}}

\DeclareMathOperator{\supp}{supp}
\DeclareMathOperator{\proj}{proj}

\newcommand{\icomp}{\mathtt{i}}
\newcommand{\abs}[1]{\left\vert#1\right\vert}
\newcommand{\norm}[1]{\left\Vert#1\right\Vert}

\allowdisplaybreaks

\mathtoolsset{showonlyrefs=true}

\journal{Nuclear Physics B}

\begin{document}

\begin{frontmatter}



\title{Worst-case $L_p$-approximation of periodic functions using median lattice algorithms\tnoteref{label0}}
\tnotetext[label0]{Dedicated to Henryk Wo\'zniakowski on the occasion of his 80th birthday.}


\author[label1]{Zexin Pan} 
\ead{zep002@zju.edu.cn}
\author[label2]{Mou Cai}
\ead{caimoumou@g.ecc.u-tokyo.ac.jp}
\author[label3]{Josef Dick}
\ead{josef.dick@unsw.edu.au}
\author[label2]{Takashi Goda}
\ead{goda@frcer.t.u-tokyo.ac.jp}
\author[label4]{Peter Kritzer}
\ead{peter.kritzer@oeaw.ac.at}


            \affiliation[label1]{organization={Institute of Fundamental and Transdisciplinary Research, Zhejiang University},
            addressline={866 Yuhangtang Road}, 
            city={Hangzhou},
            postcode={310058}, 
            state={Zhejiang},
            country={China}}
            
            \affiliation[label2]{organization={Graduate School of Engineering, The University of Tokyo},
            addressline={7-3-1 Hongo, Bunkyo-ku}, 
            city={Tokyo},
            postcode={113-8656}, 
            country={Japan}}

            \affiliation[label3]{organization={School of Mathematics and Statistics, UNSW Sydney},
            city={Kensington},
            postcode={NSW 2052}, 
            country={Australia}}

            \affiliation[label4]{organization={RICAM, Austrian Academy of Sciences},
            addressline={Altenbergerstr. 69}, 
            city={Linz},
            postcode={4040}, 
            country={Austria}}

\begin{abstract}

We study the worst-case approximation of multivariate periodic functions from the weighted Korobov space $H_{d,\alpha,\bsgamma}$ with smoothness $\alpha>1/2$ and weights $\bsgamma$ in the Lebesgue norm $L_p([0,1]^d)$ for $1\le p\le\infty$. We analyze a \emph{median lattice algorithm} that reconstructs a truncated Fourier series by approximating the coefficients on a hyperbolic-cross-type index set using $R$ rank-1 lattice sampling rules with independent randomly chosen generating vectors, and then aggregating the resulting coefficient estimators via the componentwise median. For an odd number of repetitions $R$ and an odd prime lattice size $N$, where $R$ depends logarithmically on $N$, we prove high-probability error bounds in both $L_\infty$ and $L_2$. Interpolation then yields the result for all $ p \in [1,\infty]$. In particular, with high probability, the worst-case $L_p$-error satisfies, for any $\beta>0$, an upper bound of the following form
\[
\ C_{d,\alpha,\beta,\bsgamma,\tau,p}\,R^{\frac{1}{2}+(\frac12 - \frac1p)_+} N^{- \alpha + (\frac12 - \frac1p)_+ + \beta },
\]
where $(x)_+ = \max\{x, 0\}$, and the constant $C_{d,\alpha,\beta,\bsgamma,\tau,p}$ is independent of $N$ and $R$. For $p=\infty$, $C_{d,\alpha,\beta,\bsgamma,\tau,\infty}$ is dimension-independent under the summability condition $\sum_{j=1}^\infty \gamma_j^{1/(2\alpha)}<\infty$. These results extend recent analyses of median-based lattice approximation in $L_2$ and complement related multiple-shift lattice approaches, showing that median aggregation yields nearly optimal $L_p$-approximation rates (up to logarithmic factors and an arbitrarily small loss) in weighted Korobov spaces.
\end{abstract}


\end{frontmatter}



\section{Introduction}

Rank-1 lattice rules are a classical and powerful tool in quasi-Monte Carlo (QMC) methods for high-dimensional problems. They have been developed extensively for numerical integration in periodic function spaces, with a mature theory covering existence, construction, randomization, and tractability; see, e.g., \cite{DKP22,DKS13,N92,SJ94}. In particular, for weighted Korobov spaces, component-by-component (CBC) constructions and their fast variants provide lattice rules that achieve optimal or near-optimal convergence rates under natural summability assumptions on the weights \cite{D04,DG21,DSWW06,K03,NC06,SR02}. Randomized lattice rules further yield sharp root-mean-square error bounds and often improve robustness with respect to unknown problem structure \cite{DGS22,G26,GK26,KKNU19,KNW23}.

Beyond integration, lattice sampling has become an increasingly important paradigm for \emph{approximation} and \emph{reconstruction} of multivariate periodic functions. Early investigations demonstrated how lattice information can be used to recover trigonometric polynomials and to design efficient spectral algorithms \cite{KSW06,LH03}, see also \cite{HW63}. More recent work established systematic lattice-based approximation schemes in periodic spaces with general weights together with fast CBC constructions for approximation settings \cite{CKNS20,CKNS21,KMNN21}. Complementary strands of research analyze approximation based on single or multiple rank-1 lattices from a harmonic analysis viewpoint, emphasizing sampling stability and explicit discretizations for multivariate trigonometric polynomials \cite{BKTV17,K19,KPV15,KV19}. The continuing interest in this topic is also reflected in recent advances on subsampled and structured lattices aimed at optimal $L_2$-approximation error convergence \cite{BGKS24}.

A key challenge in lattice-based approximation is to control aliasing and to obtain \emph{high-probability} guarantees for randomized constructions with minimal overhead. In the wider randomized numerical analysis and QMC literature, the \emph{median} and \emph{median-of-means} principles have emerged as effective tools for turning average-case or second-moment bounds into sharp tail estimates without requiring precise knowledge of smoothness parameters. This idea goes back to robust estimation in Monte Carlo contexts \cite{KNR19,NP09} and has recently been used to construct smoothness-adaptive or nearly universal randomized QMC strategies for integration \cite{GK26,GL22,GSM24,P26,PO23,PO24}. These developments suggest that median aggregation is a natural candidate for stabilizing lattice-based \emph{approximation} algorithms as well.

The present note fits into this emerging line of research. Building upon the $L_2$-approximation analysis of median lattice algorithms in Korobov spaces from \cite{PKG25} (see also \cite{PGK25}) and the recent multiple-shift lattice approximation framework of \cite{CDG25}, we investigate the worst-case $L_p$-approximation for the full range $1 \le p \le \infty$. We consider a randomized median lattice algorithm that approximates Fourier coefficients on a weighted hyperbolic-cross-type index set by $R$ independent rank-1 lattice rules with a prime number of points, and then applies the componentwise median across the repetitions. The main result shows that this approach yields, with high probability, nearly optimal convergence rates of the worst-case error in $L_p$, with the expected logarithmic factor associated with hyperbolic-cross truncation and an arbitrarily small loss parameter $\beta>0$. We stress that, even for the case $p=2$, the result in the present paper is different 
from those in \cite{PGK25, PKG25}, as we obtain a worst-case error bound here, as opposed to bounds for the randomized setting in those earlier papers.

In the endpoint case $p=\infty$, we additionally obtain dimension-independent constants under standard summability conditions on the product weights, mirroring the behavior known from related integration and approximation settings in weighted Korobov spaces \cite{DKP22,KKNU19,KMNN21}.

Our main result is as follows. See Section~\ref{sec_Lp} for its proof.
\begin{theorem}\label{thm:main}
Let $\alpha>1/2$ and $d \in \NN$, let $1\geq\gamma_1\geq\cdots\geq\gamma_d>0$, and use the product weights $\gamma_{\mathfrak u}=\prod_{j\in\mathfrak u}\gamma_j$, with $\gamma_\emptyset=1$.
Let $\tau>0$, let $R\geq3$ be an odd integer, let $N>2$ be an odd prime, and assume that $N_2$ defined in \eqref{eqn:Nstar} satisfies $N_2>1$. Let $\Omega=(\{1,\ldots,N-1\}^d)^R$ carry the product uniform probability measure. For $\omega=(\bsz_1,\ldots,\bsz_R)\in\Omega$, let $A^\omega$ be the approximation algorithm given by Algorithm~\ref{alg:median}. Let $\delta_2$ be given by \eqref{eqn:failprob2} and assume that $\delta_2\leq1$.
Then, for any $1 \le p \le \infty$ and any $\beta>0$,
\[
\PP_\omega\!\left(
\mathrm{err}(\Hcal_{d,\alpha,\bsgamma},L_p,A^\omega)
\leq C_{d,\alpha,\beta,\bsgamma,\tau,p}\,
R^{\frac12+(\frac12-\frac1p)_+}
N^{-\alpha+(\frac12-\frac1p)_++\beta}
\right)\geq1-\delta_2,
\]
where $(x)_+=\max\{x,0\}$ and $C_{d,\alpha,\beta,\bsgamma,\tau,p}>0$ is independent of $N$ and $R$. For $p=\infty$, this constant can also be chosen independently of $d$ if the weights are the initial segment of an infinite sequence satisfying $\sum_{j=1}^\infty\gamma_j^{1/(2\alpha)}<\infty$.
\end{theorem}

We further show in Remark~\ref{rmk:RislogN} that, for sufficiently large $N$, taking $R$ proportional to $\log N$ guarantees $\delta_2\leq 1$. Under this choice, let $n:=RN$ denote an upper bound on the number of function evaluations. The cited sampling lower bounds yield the same algebraic exponent for deterministic approximation based on $n$ function values, up to logarithmic factors and the arbitrarily small loss $n^\beta$ (see \cite{BDSU16,DKU23,T93}; see also \cite[Corollary~20]{KPUU25}). Since the source space is Hilbert, its unit ball is convex and symmetric, and point evaluations are bounded, the standard information-based complexity result in \cite[Theorem~4.7]{NW08} shows that allowing adaptivity and nonlinear reconstruction cannot improve the algebraic convergence exponent. Thus the exponent is optimal in this deterministic sampling model, up to logarithmic factors and the arbitrary $\beta$-loss.

Finally, we note that there already exist constructive algorithms that achieve the optimal worst-case $L_p$-error rate (see, e.g., \cite{T93} for $p=\infty$, \cite{BDSU16} for $p\in (2,\infty)$, and \cite{CDKU26} for $p=2$). Our algorithm, while not exactly matching the rate, offers distinct advantages in its fast implementation via the Fast Fourier Transform and its added tractability in the randomized $L^2$-error setting \cite{PKG25}. Moreover, a variant of the median lattice algorithm \cite{PGK25} requires no specification of smoothness or weight parameters, making it universally applicable across weighted Korobov spaces.

\medskip

The paper is organized as follows. Section~\ref{sec:preliminaries} introduces notation, lattice rules, weighted Korobov spaces, and the approximation error criterion. In Section~\ref{sec:median}, we describe the median lattice algorithm. Section~\ref{sec:main} contains the high-probability error analysis in $L_\infty$ and $L_2$. The proof of Theorem~\ref{thm:main} is based on the $L_2$-approximation and $L_\infty$-approximation error bounds together with a norm interpolation inequality, and is presented in Section~\ref{sec_Lp}.

\section{Preliminaries and notation}\label{sec:preliminaries}

Throughout, let \(d\in \NN\), where $\NN$ is the set of positive integers, be the dimension and let boldface symbols denote vectors in \(\RR^d\) or \(\ZZ^d\).

For a multi-index \(\bsh=(h_1,\ldots,h_d)\in\ZZ^d\), its support is denoted by
\[
\supp(\bsh):=\{\,j\in\{1,\ldots,d\}\,:\, h_j\neq 0\,\}.
\]

Let \(\bsgamma=(\gamma_{\mathfrak u})_{\mathfrak u\subseteq\{1,\ldots,d\}}\) be a family of positive weights with \(\gamma_{\emptyset}=1\). For a smoothness parameter \(\alpha>0\), define the frequency weight
\[
r_{2\alpha,\bsgamma} (\bsh):=\gamma_{\supp(\bsh)}^{-1}\prod_{j\in \supp(\bsh)} \abs{h_j}^{2\alpha},
\qquad \bsh\in\ZZ^d,
\]
so that in particular \(r_{2\alpha,\bsgamma}(\boldsymbol 0)=1\).

For \(f\in L_2([0,1]^d)\) we use the Fourier coefficients
\[
\widehat{f}(\bsh)\;:=\;\int_{[0,1]^d} f(\bsx)\,e^{-2\pi \icomp\,\bsh\cdot\bsx}\, \mathrm d\bsx,
\qquad \bsh\in\ZZ^d,
\]
where \(\bsh\cdot\bsx=\sum_{j=1}^d h_j x_j\).

The weighted periodic Korobov space \(\Hcal_{d,\alpha,\bsgamma}\) is the collection of all \(f\in L_2([0,1]^d)\) with finite norm
\[
  \|f\|_{d,\alpha,\bsgamma}^2:=\sum_{\bsh\in\ZZ^d} \abs{\widehat{f}({\bsh})}^2 r_{2\alpha,\bsgamma}(\bsh),
\]
i.e.,
\begin{equation}\label{Korspace}
\Hcal_{d,\alpha, \bsgamma} = \{ f: [0,1]^d \to \RR\,:\, \|f\|_{d, \alpha, \bsgamma} < \infty\},
\end{equation}
where the norm arises from the inner product
\[
\langle f,g\rangle_{d,\alpha,\bsgamma}\;:=\;\sum_{\bsh\in\ZZ^d} \overline{\widehat f(\bsh)}\,\widehat g(\bsh)\, r_{2\alpha,\bsgamma}(\bsh),
\]
under which $\Hcal_{d,\alpha,\bsgamma}$ is a Hilbert space.
Note that by the definition of the inner product and norm in $\Hcal_{d,\alpha,\bsgamma}$, the function $r_{2\alpha,\bsgamma}$ is linked to the decay of the Fourier coefficients of the elements of the space. Large values of $r_{2\alpha,\bsgamma}(\bsh)$ imply that the corresponding Fourier coefficients need to be small. In particular, large $\alpha$ implies fast convergence of the Fourier coefficients, which is why $\alpha$ is referred to as the smoothness parameter of $\Hcal_{d,\alpha,\bsgamma}$.

For \(g: [0,1]^d\to \RR \) we use the standard Lebesgue norms
\[
\|g\|_{L_p}:=\left(\int_{[0,1]^d}\abs{g(\bsx)}^{p}\,\mathrm d\bsx\right)^{1/p},
\qquad
\|g\|_{L_\infty}:=\operatorname*{ess\,sup}_{\bsx\in [0,1]^d}\abs{g(\bsx)}.
\]

An approximation algorithm  $A$ is a map $A:\Hcal_{d,\alpha,\bsgamma}\to L_p( [0,1]^d)$ (for $1 \le p \le \infty $) that assigns to each $f$ an approximation 
$A(f)$. The worst-case error of $A$ over the unit ball of 
$\Hcal_{d,\alpha,\bsgamma}$ is measured by
\[
  \mathrm{err}(\Hcal_{d,\alpha,\bsgamma}, L_p, A):= \sup_{\substack{f\in \Hcal_{d,\alpha,\bsgamma}\\ 
  \norm{f}_{d,\alpha,\bsgamma}\le 1}}
  \|f-A(f)\|_{L_p}, \quad 1 \le p \le \infty.
\]

In the following, we will frequently write $\{k{:}\ell\}$ to denote the set $\{k,k+1,k+2,\ldots,\ell\}$ for integers $k\le \ell$.
The role of $\bsgamma$ is to modulate the importance of coordinate subsets (see \cite{SW98}): frequencies supported on $\mathfrak u\subseteq\{1{:}d\}$ are scaled by $\gamma_{\mathfrak u}^{-1}$. In the present paper, we consider the most prominent class, namely product weights $\gamma_{\mathfrak u}=\prod_{j\in\mathfrak u}\gamma_j$ with $\gamma_{\emptyset}=1$, where $(\gamma_j)_{j\ge 1}$ is a non-increasing sequence of positive reals satisfying $1\geq\gamma_1\geq\gamma_2\geq\cdots>0$. We remark that we could also allow non-negative $\gamma_j$, but this would result in additional technical notation, which we avoid. The factor $\prod_{j\in\supp(\bsh)} |h_j|^{2\alpha}$ in the definition of $r_{2\alpha,\bsgamma}$, and thus also in the norm of $\Hcal_{d,\alpha,\bsgamma}$, penalizes higher frequencies according to the smoothness parameter $\alpha$; larger $\alpha$ corresponds to stronger decay of Fourier coefficients in $\Hcal_{d,\alpha,\bsgamma}$. 

For $\alpha>1/2$, the Cauchy--Schwarz inequality gives
\[
\sum_{\bsh\in\ZZ^d}|\widehat f(\bsh)|
\leq \|f\|_{d,\alpha,\bsgamma}
\left(\sum_{\bsh\in\ZZ^d}\frac{1}{r_{2\alpha,\bsgamma}(\bsh)}\right)^{1/2}
=\|f\|_{d,\alpha,\bsgamma}
\prod_{j=1}^d\left(1+2\gamma_j\zeta(2\alpha)\right)^{1/2}<\infty,
\]
where $\zeta (\cdot )$ is the Riemann zeta function. Hence every element of $\Hcal_{d,\alpha,\bsgamma}$ has a canonical continuous periodic representative whose Fourier series converges absolutely and uniformly. In particular, point evaluations are bounded linear functionals on $\Hcal_{d,\alpha,\bsgamma}$.

Let us now outline the definition of the QMC rules we use in our median algorithm, namely rank-1 lattice rules. For simplicity, let us assume that $N$ is a prime in the following. Choosing $N$ prime guarantees that all one-dimensional projections of a lattice point set are evenly distributed, which is an advantage. If we allowed composite $N$, we presumably could derive similar results, but would need more technical notation. 

Let $\bsz=(z_1,\ldots,z_d)$ be a 
vector with each component in $\{1{:}(N-1)\}$. We can then define a 
lattice point set $P_{N,\bsz}$ with points $\bsx_0,\ldots,\bsx_{N-1}$ as follows. For $j\in\{1{:}d\}$ and $k\in\{0{:}(N-1)\}$, the $j$-th component of $\bsx_k$ is given by
\[
x_k^{(j)}:=\left\{\frac{k z_j}{N}\right\},
\]
where $\{y\}=y-\lfloor y \rfloor$ denotes the fractional part of a real number $y$. By using this definition for all $j\in\{1{:}d\}$ and $k\in\{0{:}(N-1)\}$, we obtain the full lattice point set $P_{N,\bsz}$. A QMC rule using $P_{N,\bsz}$ is called a (rank-1) lattice rule. We further define the dual lattice of $P_{N,\bsz}$ by
\begin{align}\label{eq:def_dual_lattice}
P_{N,\bsz}^\perp := \{\bsl \in \ZZ^d \,:\,  \bsz^\top \bsl \equiv 0 \tmod N\}.
\end{align}
The dual lattice plays a very important role in the theory of lattice rules, as it can be shown that, for a lattice point set $P_{N,\bsz}$ with points $\bsx_0,\ldots,\bsx_{N-1}$ and  an integrand $g\in \Hcal_{d,\alpha,\bsgamma}$ with $\alpha>1/2$,
\begin{equation}\label{eqn:basic}
 \frac{1}{N} \sum_{k=0}^{N-1} g(\bsx_k)- \int_{[0,1]^d} g(\bsx)\rd \bsx = 
 \sum_{\bsl\in P_{N,\bsz}^\perp\setminus \{\bszero\}} \widehat{g}(\bsl),
\end{equation}
 see, e.g., \cite[Chapter 5]{N92} for a proof.

Note that, for fixed $N$ and $d$, a lattice rule is fully characterized by the choice of the \textit{generating vector} $\bsz$. Not all choices of $\bsz$ yield lattice rules of sufficient quality as integration node sets. However, there are fast construction algorithms available that return good generating vectors for given $N$, $d$, $\alpha$, and $\bsgamma$. We refer to the books \cite{DKP22, SJ94} for overviews of the theory of lattice rules, and in particular to \cite[Chapters 3 and 4]{DKP22} for constructions of good lattice rules. In the present paper, however, we will make a random choice of the generating vectors $\bsz$ to obtain the lattice point sets used. Therefore, we need not be concerned with the search for good generating vectors here, which is a general computational advantage of median rules. 

It is sometimes useful to introduce an additional random element when applying lattice rules. This is usually done by a \emph{random shift}, $\Delta\in [0,1)^d$. Given $P_{N,\bsz}=\{\bsx_0,\ldots,\bsx_{N-1}\}$ and drawing $\bsDelta$ from a uniform distribution over $[0,1)^d$, the corresponding \emph{randomly shifted lattice rule} applied to a function $g$ is then 
\[
   Q_{N,\bsz,\bsDelta}(g):=\frac{1}{N}\sum_{k=0}^{N-1}g \left(\left\{\bsx_k + \bsDelta\right\}\right),
\]
i.e., all points of $P_{N,\bsz}$ are shifted modulo one by the same $\bsDelta$.

\section{The median algorithm}\label{sec:median}

We first define a subset of $\ZZ^d$, which can be interpreted as a ``weighted'' hyperbolic cross, on which we will approximate the Fourier coefficients of a given function;
for a real $L\geq 0$, we put 
\begin{equation}\label{eq:Acal}
\Acal_{d,\alpha,\bsgamma}(L) := \left\{ \bsh \in \ZZ^d \,:\,  r_{2\alpha, \bsgamma}(\bsh) < L^{2\alpha} \right\}=\left\{ \bsh \in \ZZ^d \,:\,  r_{1, \bsgamma^{1/(2\alpha)}}(\bsh) < L\right\}.
\end{equation}

For an odd positive integer $R$, we define the median of $R$ complex numbers $Z_1,\dots,Z_R$ as
\begin{equation}\label{eqn:complexmedian}
    \mathrm{median}_{r\in\{1{:}R\}}(Z_r)=\mathrm{median}_{r\in\{1{:}R\}} \Re(Z_r) + \icomp\cdot \mathrm{median}_{r\in\{1{:}R\}} \Im(Z_r),
\end{equation}
where $\Re(Z_r)$ and $\Im(Z_r)$ denote the real and imaginary parts of $Z_r$, respectively.
In the following, we will also assume that $N$ is an odd prime 
number. Though we do not need this requirement for all results below, we sometimes do, and it is therefore easier to make this assumption from now on.
\begin{algorithm}\label{alg:median}
    Let $\alpha>1/2$ and $d \in \NN$, let $1\geq\gamma_1\geq\cdots\geq\gamma_d>0$, and use the product weights $\gamma_{\mathfrak u}=\prod_{j\in\mathfrak u}\gamma_j$. Given a parameter $\tau>0$, an odd integer $R\geq3$, and an odd prime $N>2$, do the following.
    \begin{enumerate}
        \item Define 
        \begin{equation}\label{eqn:PN}
        P_{N,d}(\tau,\alpha,\bsgamma):=\prod_{j=1}^d \left(1+ 2\gamma_j^{1/(2\alpha)} (1+\tau\log N)\right),
        \end{equation}
        and
        \begin{equation}\label{eqn:Nstar}
        N_2:=\frac{N-1 }{\exp(\tau^{-1})P_{N,d}(\tau,\alpha,\bsgamma)}.
        \end{equation}
        \item \textbf{For} $r$ from $1$ to $R$, do the following:
        \begin{enumerate}
            \item Randomly draw $\bsz_{r}$ independently from the uniform distribution over the set $\{1{:}(N-1)\}^d$.
            \item For all $\bsh\in \Acal_{d,\alpha,\bsgamma}(N_2)$, set
            \[ 
            \widehat{f}_{N,\bsz_r}(\bsh):=
            \frac{1}{N} \sum_{k=0}^{N-1} f\left( \left\{ \frac{k \bsz_r}{N} \right\}\right)\exp\left(-2\pi \icomp \bsh \cdot \left( \frac{k \bsz_r}{N}\right)\right).
            \]   
        \end{enumerate}
       \textbf{end for}
       \item For all $\bsh\in \Acal_{d,\alpha,\bsgamma}(N_2)$, define the aggregation
        \[ 
            \widehat{f}_{N,\bsz_{\{1{:}R\}}}(\bsh):=\mathrm{median}_{r\in\{1{:}R\}} \Big(\widehat{f}_{N,\bsz_r}(\bsh)\Big),
       \]
       where $\bsz_{\{1{:}R\}}$ denotes the collection of $\bsz_r$ for $r\in\{1{:}R\}$.
       Define the final approximation by
        \[ 
        (A_{N,\bsz_{\{1{:}R\}},\Acal_{d,\alpha,\bsgamma}(N_2)}(f))(\bsx):= \sum_{\bsh \in \Acal_{d,\alpha,\bsgamma}(N_2)} \widehat{f}_{N,\bsz_{\{1{:}R\}}}(\bsh) \exp (2 \pi \icomp \bsh \cdot \bsx).\]
\end{enumerate}
\end{algorithm}

For real-valued $f$, one has
$\widehat f_{N,\bsz_r}(-\bsh)=\overline{\widehat f_{N,\bsz_r}(\bsh)}$.
Since $\Acal_{d,\alpha,\bsgamma}(N_2)$ is symmetric and the componentwise median preserves this conjugate symmetry, the resulting trigonometric polynomial $A_{N,\bsz_{\{1{:}R\}},\Acal_{d,\alpha,\bsgamma}(N_2)}(f)$ is real-valued.

\begin{remark}
    The parameter $\tau>0$ does not affect the asymptotic convergence rates of Algorithm~\ref{alg:median} to be established in Section~\ref{sec:main}. For guidance on optimal tuning of $\tau$ to improve finite-sample performance, see \cite[Remark 16]{PKG25}.
\end{remark}

\begin{remark}
A natural modification of Algorithm~\ref{alg:median} is to incorporate additional independent uniform shifts $\bsDelta_1,\dots,\bsDelta_R \in [0,1)^d$, and replace $\widehat{f}_{N,\bsz_{\{1{:}R\}}}(\bsh)$ by 
       \[ 
            \widehat{f}_{N,\bsz_{\{1{:}R\}},\bsDelta_{\{1{:}R\}}}(\bsh):=\mathrm{median}_{r\in\{1{:}R\}} \Big(\widehat{f}_{N,\bsz_r,\bsDelta_r}(\bsh)\Big),
       \]
       for 
                   \[ 
            \widehat{f}_{N,\bsz_r,\bsDelta_r}(\bsh):=
            \frac{1}{N} \sum_{k=0}^{N-1} f\left( \left\{ \frac{k \bsz_r}{N}+\bsDelta_r \right\}\right)\exp\left(-2\pi \icomp \bsh \cdot \left( \frac{k \bsz_r}{N}+\bsDelta_r\right)\right).
            \]  
            For simplicity, we study Algorithm~\ref{alg:median}, but all error bounds stated in Section~\ref{sec:main} also hold under the above change due to the common upper bounds
            \begin{equation}\label{eq:dual_error}
            \abs{\widehat{f}_{N,\bsz_r}(\bsh)-\widehat{f}({\bsh})}=\abs{\sum_{\bsl\in P^\perp_{N,\bsz_r}\setminus\{\bszero\} }\widehat{f}({\bsh+\bsl})}\leq \sum_{\bsl\in P^\perp_{N,\bsz_r}\setminus\{\bszero\} }\abs{\widehat{f}({\bsh+\bsl})}     
            \end{equation}
            and 
            $$\abs{\widehat{f}_{N,\bsz_r,\bsDelta_r}(\bsh)-\widehat{f}({\bsh})}=\abs{\sum_{\bsl\in P^\perp_{N,\bsz_r}\setminus\{\bszero\} }\widehat{f}({\bsh+\bsl})\exp(2\pi\icomp \bsl\cdot\bsDelta_r)}\leq \sum_{\bsl\in P^\perp_{N,\bsz_r}\setminus\{\bszero\} }\abs{\widehat{f}({\bsh+\bsl})}.$$
      Both identities follow directly by inserting the absolutely and uniformly convergent Fourier series of $f$ and using the character orthogonality relation
\[
\frac1N\sum_{k=0}^{N-1}\exp\!\left(2\pi\icomp\,\bsg\cdot\frac{k\bsz_r}{N}\right)
=
\begin{cases}
1,&\bsg\in P_{N,\bsz_r}^\perp,\\
0,&\bsg\notin P_{N,\bsz_r}^\perp.
\end{cases}
\]
The shift contributes the additional factor $\exp(2\pi\icomp\,\bsl\cdot\bsDelta_r)$.
            
    When measured in the root-mean-square error, which is the setting considered in \cite{PKG25}, $\widehat{f}_{N,\bsz_r,\bsDelta_r}(\bsh)$ enjoys a tighter error bound
    \begin{equation*}
      \left(\int_{[0,1]^d} |\widehat{f}_{N,\bsz_r,\bsDelta_r}(\bsh)-\widehat{f}(\bsh)|^2 \rd\bsDelta_r\right)^{1/2}= \left(\sum_{\bsl \in P_{N,\boldsymbol{z}_r}^\perp\setminus \{\bszero\}}|\widehat{f}(\bsh+\bsl)|^2\right)^{1/2}.
    \end{equation*}
    Therefore, we expect median lattice algorithms with the random shifts to perform better than those without the shifts, despite the two sharing the same worst-case error bound.
\end{remark}

\section{Approximation error bounds}\label{sec:main}

Throughout this section, we assume that $\tau>0$, $R\geq3$ is an odd integer, and $N>2$ is a sufficiently large odd prime such that $N_2>1$.

\subsection{$L_\infty$-approximation}

The following lemma describes the scaling of $N_2$ with respect to $N$. In particular, for any $\beta>0$, applying the lemma with $\beta/2$ and increasing the finite threshold if necessary gives $N_2\in[N^{1-\beta},N]$ for all sufficiently large $N$. Consequently, we may replace $N_2$ by $N$ in the error bound at the cost of an arbitrarily small loss in the exponent.

\begin{lemma}\label{lem:Nstar}
    For $N_2$ defined as in \eqref{eqn:Nstar}, there exists a constant $c_{d,\alpha,\beta,\bsgamma,\tau}>0$ independent of $N$ such that
    $$c_{d,\alpha,\beta,\bsgamma,\tau}N^{1-\beta}\leq N_2\leq N.$$
    Furthermore, the constant $c_{d,\alpha,\beta,\bsgamma,\tau}$ can be chosen independent of the dimension $d$ if the weights $\bsgamma$ satisfy $\sum_{j=1}^\infty \gamma_j^{1/(2\alpha)}<\infty$.
\end{lemma}
\begin{proof}
By \eqref{eqn:PN}, $P_{N,d}(\tau,\alpha,\bsgamma)$ is a polynomial in $\log N$, and hence $P_{N,d}(\tau,\alpha,\bsgamma)=O(N^{\beta})$ for any $\beta>0$. The first claim then follows directly from \eqref{eqn:Nstar}. The second claim follows from \cite[Remark 14]{PKG25}, whose proof is a minor adaptation of that of \cite[Lemma 3]{HN03}.
\end{proof}

The next lemma provides a bound on the size of the frequency set $\Acal_{d,\alpha, \bsgamma}(N_2)$. Results of this type are well known; we use a bound from \cite[Corollary 6]{PKG25}.
\begin{lemma}\label{lem:Adcount} For $N_2$ defined as in \eqref{eqn:Nstar} and $\Acal_{d,\alpha,\bsgamma}(N_2)$ as in \eqref{eq:Acal}, 
we have
\[ 
 |\Acal_{d,\alpha,\bsgamma}(N_2)|\leq 1+\frac{ N-1}{1+\tau\log N_2}.
\]
\end{lemma}

For a given generating vector $\bsz$, we define the set $K_{d, \alpha, \bsgamma}(\bsz)$ of frequencies $\bsh \in \Acal_{d, \alpha, \bsgamma}$ for which no frequency $\bsk \in \Acal_{d, \alpha, \bsgamma}$ is aliased with $\bsh$, i.e., we have $(\bsh - \bsk) \cdot \bsz \not\equiv 0 \pmod{N}$ for all $\bsk \in \Acal_{d, \alpha, \bsgamma}\setminus \{\bsh\}$: 
\begin{equation*}
K_{d,\alpha,\bsgamma}(\bsz):=\left\{\bsh\in \Acal_{d,\alpha,\bsgamma}(N_2)\,:\,   (\bsh+ P^\perp_{N,\bsz}) \bigcap \Acal_{d,\alpha,\bsgamma}(N_2)=\{\bsh\} \right\}.    
\end{equation*}
This set is crucial in the proof of our main result, since for ``good'' generating vectors $\bsz$ we obtain ``good'' approximations of the Fourier coefficients for frequencies in $K_{d,\alpha, \bsgamma}(\bsz)$. Roughly speaking, the proof strategy is to show that for each frequency $\bsh \in \Acal_{d,\alpha, \bsgamma}(N_2)$, more than half of the $\bsz_1, \ldots, \bsz_R$ are such that $\bsh \in K_{d,\alpha, \bsgamma}(\bsz_r)$.

Next, we show several auxiliary results.
\begin{lemma}\label{lem:reconstruction}
    For $\bsh\in \Acal_{d,\alpha,\bsgamma}(N_2) $ and $\bsz$ chosen uniformly from  $\{1{:}(N-1)\}^d$, 
    $$\Pr(\bsh\notin K_{d,\alpha,\bsgamma}(\bsz) )\leq \frac{1}{1+\tau\log N_2}.$$
\end{lemma}
\begin{proof}
    By Lemma~\ref{lem:Adcount}, we have $|\Acal_{d,\alpha,\bsgamma}(N_2)|<N$. Fix $\bsh\in\Acal_{d,\alpha,\bsgamma}(N_2)$ and a coordinate $j$, and put $m:=|h_j|$. Replacing $h_j$ by any integer $n$ with $|n|\leq m$ cannot increase $r_{2\alpha,\bsgamma}$, so every integer in $\{-m,\ldots,m\}$ belongs to the projection
    \[
    \proj_j(\Acal_{d,\alpha,\bsgamma}(N_2))
    :=\{k_j\in\ZZ:\bsk\in\Acal_{d,\alpha,\bsgamma}(N_2)\}.
    \]
    Hence $2m+1\leq|\proj_j(\Acal_{d,\alpha,\bsgamma}(N_2))|\leq|\Acal_{d,\alpha,\bsgamma}(N_2)|<N$, and therefore $|h_j|\leq(N-1)/2$. Thus, for distinct $\bsh,\bsh'\in\Acal_{d,\alpha,\bsgamma}(N_2)$, at least one coordinate difference $h_j-h'_j$ is nonzero and has absolute value smaller than $N$. Consequently, $\bsh\not\equiv\bsh'\pmod N$.

    Now, for each $\bsh'\in \Acal_{d,\alpha,\bsgamma}(N_2) $ distinct from $\bsh$, the fact that $N\nmid \bsh-\bsh'$ leads to
    $$\Pr\left(\bsz\cdot (\bsh'-\bsh)\equiv 0 \tmod N  \right)\leq \frac{1}{N-1}.$$
    Hence, $\Pr(\bsh'-\bsh\in P^\perp_{N,\bsz})\leq 1/(N-1)$. By a union bound, we obtain
    \[
    \Pr(\bsh\notin K_{d,\alpha,\bsgamma}(\bsz))\leq \frac{| \Acal_{d,\alpha,\bsgamma}(N_2)|-1}{N-1}\leq\frac{1}{1+\tau\log N_2}.\qedhere
    \]
\end{proof}

\begin{corollary}\label{cor:whp}
    Let 
    \begin{equation}\label{eqn:failprob}
      \delta_1:=\frac{|\Acal_{d,\alpha,\bsgamma}(N_2)|}{2}\left(\frac{4}{1+\tau\log N_2}\right)^{\lceil R/2 \rceil}.  
    \end{equation}
    If $ \delta_1\leq 1$, then with probability at least $1- \delta_1$, 
    $$
    \min_{\bsh\in \Acal_{d,\alpha,\bsgamma}(N_2)} \abs{\{r\in\{1{:}R\}\colon \bsh\in K_{d,\alpha,\bsgamma}(\bsz_r)\}}\geq \lceil R/2 \rceil,
    $$
    Here $|\cdot|$ denotes cardinality when applied to a finite set.
\end{corollary}
\begin{proof}
    For each $\bsh\in \Acal_{d,\alpha,\bsgamma}(N_2)$, Lemma~\ref{lem:reconstruction} implies that the probability of $\bsh$ belonging to none of 
$K_{d,\alpha,\bsgamma}(\bsz_1),\dots,K_{d,\alpha,\bsgamma}(\bsz_{\lceil R/2 \rceil})$ is at most $(1+\tau\log N_2)^{-\lceil R/2 \rceil}$. 
The same bound holds for every $\lceil R/2 \rceil$-element subset of 
$\{K_{d,\alpha,\bsgamma}(\bsz_1),\dots,K_{d,\alpha,\bsgamma}(\bsz_R)\}$. Since there are at most $2^R$ such subsets, a union bound gives that the probability of $\bsh$ being excluded from at least $\lceil R/2 \rceil$ of these sets is bounded by
$$2^R\left(\frac{1}{1+\tau\log N_2}\right)^{\lceil R/2 \rceil}
= \frac{1}{2} \left(\frac{4}{1+\tau\log N_2}\right)^{\lceil R/2 \rceil}.$$
The conclusion then follows from a further union bound over $\bsh\in \Acal_{d,\alpha,\bsgamma}(N_2)$.
\end{proof}
\begin{lemma}\label{lem:medtosum}
    Let $\{x_1,\dots,x_R\}\subseteq [0,\infty)$. For any $J\subseteq \{1{:}R\}$ with $|J|\geq \lceil R/2 \rceil$,
    $$\mathrm{median}_{r\in\{1{:}R\}} x_r \leq \sum_{r\in J} x_r.$$
\end{lemma}
\begin{proof}
    Let $r^*\in \{1{:}R\}$ be such that it satisfies $x_{r^*}=\mathrm{median}_{r\in\{1{:}R\}} x_r$. If $r^*\in J$, the inequality is trivial. When $r^*\notin J$, we must have
    \[
    x_{r^*}\leq \max_{r\in J} x_r\leq \sum_{r\in J} x_r.  \qedhere
\]
    
\end{proof}
We need the following lemma, whose proof can be found, for instance, in \cite[Chapter~14, Lemma~14.2]{DKP22}.
\begin{lemma}\label{lem_sum}
Let $\alpha>1/2$ and $L\geq 1$ be given. For any $q\in(1/(2\alpha),1)$ we have 
\begin{align*}
\sum_{\bsk\in\mathbb{Z}^d\backslash \Acal_{d, \alpha,\bsgamma}(L)}\frac{1}{r_{2\alpha,\bsgamma}(\bsk)}\leq \frac{1}{(\gamma_1^{1/2} L^{2\alpha})^{(1/q-1)/(2\alpha)}}\frac{q}{1-q}\prod_{j=1}^d\left(1+2\gamma_j^{q} \zeta(2\alpha q)  \right)^{1/q},
\end{align*}
where $\zeta (\cdot )$ is the Riemann zeta function.
\end{lemma}

\begin{theorem}\label{thm:infty}
Let $\alpha>1/2$. Let $\delta_1$ be given by \eqref{eqn:failprob} and assume it 
is at most 1. 
Then, for any $q\in(1/(2\alpha),1)$, with probability at least $1-\delta_1$,
\begin{multline*}
 \mathrm{err}(\Hcal_{d,\alpha,\bsgamma}, L_\infty, A_{N,\bsz_{\{1{:}R\}},\Acal_{d,\alpha,\bsgamma}(N_2)})\\
 \leq ( 2R+1) \left( \frac{1}{(\gamma_1^{1/(4\alpha)} N_2)^{(1/q-1)}}\frac{q}{1-q}\prod_{j=1}^d\left(1+2\gamma_j^{q} \zeta(2\alpha q)  \right)^{1/q} \right)^{1/2}.   
\end{multline*}
\end{theorem}

\begin{proof}
    For $f\in \Hcal_{d,\alpha,\bsgamma}$,
    $$\Vert A_{N,\bsz_{\{1{:}R\}},\Acal_{d,\alpha,\bsgamma}(N_2)}(f)-f\Vert_{L_\infty}\leq \sum_{\bsh\in \Acal_{d,\alpha,\bsgamma}(N_2)} \abs{\widehat{f}_{N,\bsz_{\{1{:}R\}}}(\bsh)-\widehat{f}({\bsh})}+\sum_{\bsh\in \ZZ^d\setminus\Acal_{d,\alpha,\bsgamma}(N_2)} \abs{\widehat{f}({\bsh})}.$$

In the following we use that for a set of real numbers $x_1, x_2, \ldots, x_R$, with $R$ odd, we have $| \mathrm{median}_{r \in \{1:R\}} x_r | \le \mathrm{median}_{r \in \{1:R\}} |x_r|$. Then the complex median \eqref{eqn:complexmedian} satisfies
\begin{align*}
    &\abs{\widehat{f}_{N,\bsz_{\{1{:}R\}}}(\bsh)-\widehat{f}({\bsh})}\\
    &\leq \mathrm{median}_{r\in\{1{:}R\}} \abs{\Re\left(\widehat{f}_{N,\bsz_{r}}(\bsh)-\widehat{f}({\bsh})\right)}+\mathrm{median}_{r\in\{1{:}R\}} \abs{\Im\left(\widehat{f}_{N,\bsz_{r}}(\bsh)-\widehat{f}({\bsh})\right)}\\
    &\leq  2 \ \mathrm{median}_{r\in\{1{:}R\}} \abs{\widehat{f}_{N,\bsz_{r}}(\bsh)-\widehat{f}({\bsh})}.
\end{align*}
By Corollary~\ref{cor:whp} and Lemma~\ref{lem:medtosum}, with probability at least $1-\delta_1$, it holds that
    \begin{align*}
        \sum_{\bsh\in \Acal_{d,\alpha,\bsgamma}(N_2)} \abs{\widehat{f}_{N,\bsz_{\{1{:}R\}}}(\bsh)-\widehat{f}({\bsh})} & \leq 2\sum_{\bsh\in \Acal_{d,\alpha,\bsgamma}(N_2)}  \mathrm{median}_{r\in\{1{:}R\}} \abs{\widehat{f}_{N,\bsz_r}(\bsh)-\widehat{f}({\bsh})} \\
        & \overset{\text{(i)}}{\leq} 2 \sum_{r=1}^R \sum_{\bsh\in K_{d,\alpha,\bsgamma}(\bsz_r)} \abs{\widehat{f}_{N,\bsz_r}(\bsh)-\widehat{f}({\bsh})}\\
        & \overset{\text{(ii)}}{\leq} 2 \sum_{r=1}^R \sum_{\bsh\in K_{d,\alpha,\bsgamma}(\bsz_r)} \sum_{\bsl\in P^\perp_{N,\bsz_r}\setminus\{\bszero\} }\abs{\widehat{f}({\bsh+\bsl})}\\
        & \overset{\text{(iii)}}{\leq} 2R \sum_{\bsh\in \ZZ^d\setminus\Acal_{d,\alpha,\bsgamma}(N_2)} \abs{\widehat{f}({\bsh})},
    \end{align*}
    where (i) uses Lemma~\ref{lem:medtosum}, (ii) uses \eqref{eq:dual_error}, and (iii) follows because  $(\bsh,\bsl) \to \bsh+\bsl$ is injective over $K_{d,\alpha,\bsgamma}(\bsz_r)\times  P^\perp_{N,\bsz_r} $, and $\bsh+\bsl\in \Acal_{d,\alpha,\bsgamma}(N_2)$ holds only if $\bsl=\bszero$. Therefore, using the Cauchy-Schwarz inequality,
    \begin{align*}
        \Vert A_{N,\bsz_{\{1{:}R\}},\Acal_{d,\alpha,\bsgamma}(N_2)}(f)-f\Vert_{L_\infty} & \leq  (2 R+1)\sum_{\bsh\in \ZZ^d\setminus\Acal_{d,\alpha,\bsgamma}(N_2)} \abs{\widehat{f}({\bsh})} \\
        & \leq (2 R+1) \|f\|_{d,\alpha,\bsgamma} \left(\sum_{\bsh\in \ZZ^d\setminus\Acal_{d,\alpha,\bsgamma}(N_2)} \frac{1}{r_{2\alpha,\bsgamma}(\bsh)}\right)^{1/2}.
    \end{align*}
    The conclusion follows immediately from Lemma~\ref{lem_sum}.
\end{proof}

When $R$ scales logarithmically with $N$ (see Remark~\ref{rmk:RislogN}), the worst-case $L^\infty$-error is $O(N^{-\alpha+1/2+\beta})$ for any $\beta>0$ with probability $1-\delta_1$, since $q$ can be arbitrarily close to $1/(2\alpha)$ and $N_2\geq N^{1-\beta}$ for any $\beta>0$ and sufficiently large $N$. Furthermore, the bound is independent of the dimension $d$ if $\sum_{j=1}^\infty \gamma_j^{1/(2\alpha)} < \infty$. Indeed, consider the estimate
\begin{eqnarray*}
  \prod_{j=1}^d\left(1+2\gamma_j^{q} \zeta(2\alpha q)  \right)^{1/q} = 
  \exp \left(\sum_{j=1}^d \frac{1}{q} \log \left(1+2\gamma_j^{q} \zeta(2\alpha q)  \right)\right)
  \le \exp \left(\frac{2\zeta(2\alpha q)}{q}\sum_{j=1}^\infty \gamma_j^{q}\right),
\end{eqnarray*}
where we used that $\log (1+x) \le x$ for nonnegative $x$. Since $q>1/(2\alpha)$ and $\gamma_j\le 1$, we have $\gamma_j^q \le \gamma_j^{1/(2\alpha)}$, so $\sum_{j=1}^\infty \gamma_j^q \le \sum_{j=1}^\infty \gamma_j^{1/(2\alpha)} < \infty$, 
which gives the dimension-independent bound when combined with Lemma~\ref{lem:Nstar}.

\begin{remark}
    In the proof of Theorem~\ref{thm:infty}, we actually obtain the following tighter bound:
    $$\Vert A_{N,\bsz_{\{1{:}R\}},\Acal_{d,\alpha,\bsgamma}(N_2)}(f)-f\Vert_{L_\infty}  \leq  (2 R+1)\sum_{\bsh\in \ZZ^d\setminus\Acal_{d,\alpha,\bsgamma}(N_2)} \abs{\widehat{f}({\bsh})}.$$
    Consequently, for functions $f$ in the Wiener space
    $$\mathcal{W}_{d,\alpha,\gamma}=\{ f: [0,1]^d \to \RR\,:\,  \|f\|_{\mathcal{W},d, \alpha, \bsgamma} < \infty\} \quad\text{for}\quad \|f\|_{\mathcal{W},d, \alpha, \bsgamma}=\sum_{\bsh\in\ZZ^d} \abs{\widehat{f}({\bsh})} \sqrt{r_{2\alpha,\bsgamma}(\bsh)},$$
    the median lattice algorithm  satisfies, with probability at least $1-\delta_1$, 
    $$\sup_{\substack{f\in \mathcal{W}_{d,\alpha,\bsgamma}\\ 
  \norm{f}_{\mathcal{W},d,\alpha,\bsgamma}\le 1}}\Vert A_{N,\bsz_{\{1{:}R\}},\Acal_{d,\alpha,\bsgamma}(N_2)}(f)-f\Vert_{L_\infty} \leq  (2 R+1) N_2^{-\alpha}.$$
  Again, when $R$ scales logarithmically with $N$, this convergence rate
  matches the optimal one up to $\log N$ factors \cite{NNS22}. 
\end{remark}

\subsection{$L_2$-approximation}

For a generating vector $\bsz$, define
$$\rho_{2\alpha, \bsgamma}(\bsz)=\min_{\bsl\in P^\perp_{N,\bsz}\setminus\{\bszero\}} r_{2\alpha,\bsgamma}(\bsl).$$

For $L_\infty$-approximation it was sufficient to show that for each $\bsh \in \Acal_{d,\alpha, \bsgamma}(N_2)$, more than half of the generating vectors satisfy $\bsh \in K_{d,\alpha, \bsgamma}(\bsz_r)$, which was shown in Corollary~\ref{cor:whp}. For $L_2$-approximation we additionally need, for each reconstructed frequency, at least half of the repetitions to satisfy both $\bsh\in K_{d,\alpha,\bsgamma}(\bsz_r)$ and $\rho_{2\alpha,\bsgamma}(\bsz_r)\geq N_2^{2\alpha}$. The following lemma establishes this joint property. It allows us to improve the convergence rate of the $L_2$-approximation error by a factor of $N_2^{1/2}$ compared with the $L_\infty$-approximation error.

\begin{lemma}\label{cor:whp2}
    Let 
    \begin{equation}\label{eqn:failprob2}
      \delta_2:=\frac{|\Acal_{d,\alpha,\bsgamma}(N_2)|}{2}\left(\frac{8}{1+\tau\log N_2}\right)^{\lceil R/2 \rceil}.  
    \end{equation}
    If $\delta_2\leq 1$, then with probability at least $1- \delta_2$, 
    \[
      \min_{\bsh\in \Acal_{d,\alpha,\bsgamma}(N_2)} \abs{\{r\in \mathcal{R}\colon \bsh\in K_{d,\alpha,\bsgamma}(\bsz_r)\}}\geq \lceil R/2 \rceil
    \]
    for 
    \begin{equation}\label{eqn:calR}
            \mathcal{R}=\left\{r\in \{1{:}R\} \,:\, \rho_{2\alpha, \bsgamma}(\bsz_r)\geq N^{2\alpha}_2 \right\}.  
    \end{equation}
\end{lemma}
\begin{proof}
For any $r\in \{1{:}R\}$, when $\bszero\in K_{d,\alpha,\bsgamma}(\bsz_r)$, no $\bsl\in P^\perp_{N,\bsz_r}$ other than $\bszero$ belong to $\Acal_{d,\alpha,\bsgamma}(N_2)$, so
$$\rho_{2\alpha, \bsgamma}(\bsz_r)=\min_{\bsl\in P^\perp_{N,\bsz_r}\setminus\{\bszero\}} r_{2\alpha,\bsgamma}(\bsl)\geq N^{2\alpha}_2$$
and $r\in\mathcal{R}$. Consequently, for any $\bsh\in  \Acal_{d,\alpha,\bsgamma}(N_2)$,
\[
  \abs{\{r\in \mathcal{R}\colon \bsh\in K_{d,\alpha,\bsgamma}(\bsz_r)\}} \ge 
  \abs{\{r\in \{1:R\}\colon \bszero\in K_{d,\alpha,\bsgamma}(\bsz_r)\text{ and }\bsh\in K_{d,\alpha,\bsgamma}(\bsz_r)\}}.
\]
The probability that
\[
\abs{\{r\in \{1:R\}\colon \bszero\in K_{d,\alpha,\bsgamma}(\bsz_r)\text{ and }\bsh\in K_{d,\alpha,\bsgamma}(\bsz_r)\}} < \lceil R/2 \rceil 
\]
for at least one $\bsh\in \Acal_{d,\alpha,\bsgamma}(N_2)$ is equivalently the probability that 
\begin{equation}\label{eq:probability_converse}
 \abs{\{r\in \{1:R\}\colon \bszero\notin K_{d,\alpha,\bsgamma}(\bsz_r)\text{ or }\bsh\notin K_{d,\alpha,\bsgamma}(\bsz_r)\}} \ge \lceil R/2 \rceil .
\end{equation}
for at least one $\bsh\in \Acal_{d,\alpha,\bsgamma}(N_2)$.
However, Lemma~\ref{lem:reconstruction} implies for every $r\in \{1{:}R\}$ and $\bsh\in  \Acal_{d,\alpha,\bsgamma}(N_2)$,
$$\Pr\left(\bszero\notin K_{d,\alpha,\bsgamma}(\bsz_r)\text{ or } \bsh\notin K_{d,\alpha,\bsgamma}(\bsz_r)\right)\leq \frac{2}{1+\tau\log N_{2}}.$$
Therefore,  by an argument similar to that in the proof of Corollary~\ref{cor:whp}, the probability that \eqref{eq:probability_converse} holds for a given $\bsh\in \Acal_{d,\alpha,\bsgamma}(N_2)$ is bounded from above by 
$$2^R \left(\frac{2}{1+\tau\log N_{2}}\right)^{\lceil R/2 \rceil}\leq \frac{1}{2}\left(\frac{8}{1+\tau\log N_{2}}\right)^{\lceil R/2 \rceil}.$$
The conclusion follows after taking a union bound over $\bsh\in  \Acal_{d,\alpha,\bsgamma}(N_2)$.
\end{proof}

The following lemma bounds the maximum intersection of the dual lattice $P_{N,\bsz}^\perp$ with $\mathcal{A}_{d,\alpha,\gamma}(L)$ under an arbitrary shift $\bsh\in \mathbb{Z}^d$.

\begin{lemma}\label{lem:fiber}
Let $\alpha>1/2$, $\gamma_j\in(0,1]$ for $j\in\{1{:}d\}$, use the product weights $\gamma_{\mathfrak u}=\prod_{j\in\mathfrak u}\gamma_j$, let \(N>2\) be prime, and let $\bsz\in\{1{:}(N-1)\}^d$ satisfy 
$\rho_{2\alpha,\bsgamma}(\bsz) > 1$.
For \(L>1\), let \(m\in\mathbb{N}\) be such that $2^{m-1}<L\le 2^m$.
Then, for every \(\bsh\in\mathbb{Z}^d\),
\[
\left|(\bsh+P_{N,\bsz}^\perp)\cap \mathcal{A}_{d,\alpha,\bsgamma}(L)\right|
\le 3^d (m+d)^d\left(1+4^d\frac{2^{m}}{\rho_{2\alpha,\bsgamma}(\bsz)^{1/(2\alpha)}}\right).
\]
\end{lemma}

\begin{proof}
We follow the proof of \cite[Lemma 4]{CDG25}.
Write $L_0:=\rho_{2\alpha,\bsgamma}(\bsz)^{1/(2\alpha)}>1$. For \(\bsk\in\mathbb{Z}^d\), define \(t_j(\bsk):=|k_j|/\gamma_j^{1/(2\alpha)}\).
By definition of \(\rho_{2\alpha,\bsgamma}(\bsz)\), every nonzero \(\boldsymbol{\ell}\in P_{N,\bsz}^\perp\) satisfies
\begin{equation}\label{eqn:L0}
    \prod_{j\in\operatorname{supp}(\boldsymbol{\ell})} t_j(\boldsymbol{\ell})\ge L_0. 
\end{equation}

Write \(S:=(\bsh+P_{N,\bsz}^\perp)\cap \mathcal{A}_{d,\alpha,\bsgamma}(L)\).
For \(\boldsymbol{\varepsilon}\in\{-1,0,1\}^d\), let
\[
S_{\boldsymbol{\varepsilon}}:=\{\bsk\in S: \operatorname{sgn}(k_j)=\varepsilon_j\text{ for every }j\},
\]
Thus every point of \(S\) belongs to exactly one \(S_{\boldsymbol{\varepsilon}}\). 

Fix \(\boldsymbol{\varepsilon}\) with nonempty support \(u:=\{j:\varepsilon_j\ne 0\}\). For \(j\in u\), choose integers \(n_j\ge 1\) such that
\[
2^{n_j-1}\le t_j(\bsk)< 2^{n_j}.
\]
Since \(\bsk\in\mathcal{A}_{d,\alpha,\bsgamma}(L)\), we have
\[
\prod_{j\in u}t_j(\bsk)<L\le 2^m,
\]
so
\[
\sum_{j\in u}(n_j-1)\le m-1.
\]
The number of possible index vectors \((n_j)_{j\in u}\) is at most
\[
\binom{m+|u|-1}{|u|}\le (m+d)^{d}.
\]

For a fixed index vector \((n_j)_{j\in u}\), consider the dyadic box
\[
B=\prod_{j\in u}[2^{n_j-1},2^{n_j}).
\]
For two distinct points \(\bsk,\bsk'\in S_\varepsilon\cap B\), their nonzero coordinates have the same signs, so
\[
|t_j(\bsk)-t_j(\bsk')|=\frac{|k_j-k_j'|}{\gamma_j^{1/(2\alpha)}}.
\]
By \eqref{eqn:L0}, for \(v=\operatorname{supp}(\bsk-\bsk')\),
\begin{equation}\label{eqn:separation}
    \prod_{j\in v}|t_j(\bsk)-t_j(\bsk')|\ge L_0.
\end{equation}

We now bound the number of points in \(S_{\boldsymbol{\varepsilon}}\cap B\). Let \(R_j:=2^{n_j-1}\) for \(j\in u\) and let $\operatorname{vol}(B):=\prod_{j\in u}R_j$. If \(\operatorname{vol}(B)>L_0/2\), put
\[
D:=\max\left\{1,\frac{L_0}{2}\right\}.
\]
Then $1\leq D\leq\operatorname{vol}(B)$, so we can choose 
\(\delta_j\in[1,R_j]\), for $j\in u$, such that
\[
\prod_{j\in u}\delta_j=D.
\]
Partition each interval \([R_j,2R_j)\) into $\lceil R_j/\delta_j\rceil$ half-open subintervals of equal length. Each such interval has length at most $\delta_j$, and $\lceil R_j/\delta_j\rceil\leq2R_j/\delta_j$. If two points of \(S_{\boldsymbol{\varepsilon}}\) lie in the same subbox, then, for $v=\operatorname{supp}(\bsk-\bsk')$,
\[
\prod_{j\in v}|t_j(\bsk)-t_j(\bsk')|
\leq\prod_{j\in v}\delta_j
\leq\prod_{j\in u}\delta_j
=D<L_0,
\]
where the last inequality follows from $L_0>1$. This contradicts the separation 
property \eqref{eqn:separation}. Hence each subbox contains at most one point, and the number of subboxes is at most
\[
2^{|u|}\frac{\operatorname{vol}(B)}{D}
\leq4^{|u|}\frac{2^m}{L_0}.
\]
Indeed, if $L_0\geq2$, then $D=L_0/2$ and $\operatorname{vol}(B)\leq2^{m-1}$; if $1<L_0<2$, then $D=1$ and the same bound follows from $\operatorname{vol}(B)\leq2^{m-1}$. If \(\operatorname{vol}(B)\leq L_0/2\), two distinct points in $S_\varepsilon\cap B$ would satisfy
\[
\prod_{j\in v}|t_j(\bsk)-t_j(\bsk')|
<\prod_{j\in v}R_j\leq\operatorname{vol}(B)\leq\frac{L_0}{2},
\]
again contradicting the separation property. Thus $|S_{\boldsymbol{\varepsilon}}\cap B|\leq1$ in this case. Therefore, in all cases,
\[
|S_{\boldsymbol{\varepsilon}}\cap B|\leq1+4^{|u|}\frac{2^m}{L_0}
\leq1+4^d\frac{2^m}{L_0}.
\]

Summing over the sign vectors and dyadic index vectors, we obtain
\[
|S|
\le 1+\sum_{\substack{\bszero\neq \boldsymbol{\varepsilon}\in \{-1,0,1\}^d\\}}(m+d)^{d}\left(1+4^{d}\frac{2^{m}}{L_0}\right)\leq 3^d (m+d)^d\left(1+4^d\frac{2^{m}}{L_0}\right).
\]
Recalling that $L_0=\rho_{2\alpha,\bsgamma}(\bsz)^{1/(2\alpha)}$ completes the proof.
\end{proof}

\begin{theorem}\label{thm:2}
Let $\alpha>1/2$. Let $\delta_2$ be given by \eqref{eqn:failprob2} and assume it is at most 1. Then with probability at least $1-\delta_2$,
$$    \mathrm{err}(\Hcal_{d,\alpha,\bsgamma}, L_2,  A_{N,\bsz_{\{1{:}R\}},\Acal_{d,\alpha,\bsgamma}(N_2)}) \leq     C_{d, \alpha}  \sqrt{2R+1}\frac{ (1+\log_2 N_2)^{\frac{d}{2}}}{N_2^{ \alpha}}, $$ where $C_{d, \alpha} > 0$ is a constant depending only on $d$ and $\alpha$.
\end{theorem}

\begin{proof}
     For $f\in \Hcal_{d,\alpha,\bsgamma}$, we have
    \begin{equation}\label{eqn:L2errordecomp}
         \Vert A_{N,\bsz_{\{1{:}R\}},\Acal_{d,\alpha,\bsgamma}(N_2)}(f)-f\Vert^2_{L_2} = \sum_{\bsh\in \Acal_{d,\alpha,\bsgamma}(N_2)} \abs{\widehat{f}_{N,\bsz_{\{1{:}R\}}}(\bsh)-\widehat{f}({\bsh})}^2+\sum_{\bsh\in \ZZ^d\setminus\Acal_{d,\alpha,\bsgamma}(N_2)} \abs{\widehat{f}({\bsh})}^2.
    \end{equation}
    The second term on the right hand side can be bounded by
    \begin{equation*}
        \sum_{\bsh\in \ZZ^d\setminus\Acal_{d,\alpha,\bsgamma}(N_2)} \abs{\widehat{f}({\bsh})}^2\leq \frac{1}{N^{2\alpha}_2}\sum_{\bsh\in \ZZ^d\setminus\Acal_{d,\alpha,\bsgamma}(N_2)} \abs{\widehat{f}({\bsh})}^2r_{2\alpha,\bsgamma}(\bsh)\leq  \frac{\|f\|^2_{d,\alpha,\bsgamma}}{N^{2\alpha}_2}.
    \end{equation*}
    The componentwise median is translation-equivariant. Moreover, for complex numbers $e_1,\ldots,e_R$, it satisfies
    \[
    \left|\operatorname{median}_{r\in\{1{:}R\}}e_r\right|^2
    \leq \operatorname{median}_{r}|\Re(e_r)|^2
       +\operatorname{median}_{r}|\Im(e_r)|^2
    \leq2\operatorname{median}_{r}|e_r|^2.
    \]
    For the first term, Lemma~\ref{cor:whp2} shows that with probability at least $1-\delta_2$, 
        \begin{align*}
         & \sum_{\bsh\in \Acal_{d,\alpha,\bsgamma}(N_2)} \abs{\widehat{f}_{N,\bsz_{\{1{:}R\}}}(\bsh)-\widehat{f}({\bsh})}^2 \\
         & \leq 2\sum_{\bsh\in \Acal_{d,\alpha,\bsgamma}(N_2)}  \mathrm{median}_{r\in\{1{:}R\}} \abs{\widehat{f}_{N,\bsz_r}(\bsh)-\widehat{f}({\bsh})}^2 \\
         & \overset{\text{(i)}}{\leq} 2 \sum_{r\in\mathcal{R}} \sum_{\bsh\in K_{d,\alpha,\bsgamma}(\bsz_r)} \abs{\widehat{f}_{N,\bsz_r}(\bsh)-\widehat{f}({\bsh})}^2\\
         & \overset{\text{(ii)}}{\leq} 2  \sum_{r\in\mathcal{R}} \sum_{\bsh\in K_{d,\alpha,\bsgamma}(\bsz_r)} \left(\sum_{\bsl\in P^\perp_{N,\bsz_r}\setminus\{\bszero\} } \abs{\widehat{f}({\bsh+\bsl})}\right)^2 \\
         & \overset{\text{(iii)}}{\leq} 2  \sum_{r\in\mathcal{R}} \sum_{\bsh\in K_{d,\alpha,\bsgamma}(\bsz_r)} \left(\sum_{\bsl\in P^\perp_{N,\bsz_r}\setminus\{\bszero\} } \abs{\widehat{f}({\bsh+\bsl})}^2 r_{2\alpha,\bsgamma}(\bsh+\bsl) \right) \left(\sum_{\bsl\in P^\perp_{N,\bsz_r}\setminus\{\bszero\} }\frac{1}{r_{2\alpha,\bsgamma}(\bsh+\bsl)} \right)\\
         & \leq 2 R \|f\|^2_{d,\alpha,\bsgamma}  \sup_{r \in \mathcal{R} }\sup_{\bsh\in K_{d,\alpha,\bsgamma}(\bsz_r)} \sum_{\bsl\in P^\perp_{N,\bsz_r}\setminus\{\bszero\} }\frac{1}{r_{2\alpha,\bsgamma}(\bsh+\bsl)},
    \end{align*}
    Here (i) uses Lemma~\ref{lem:medtosum}, (ii) uses \eqref{eq:dual_error}, and (iii) uses the Cauchy--Schwarz inequality. For fixed $r$, the map $(\bsh,\bsl)\mapsto\bsh+\bsl$ is injective on $K_{d,\alpha,\bsgamma}(\bsz_r)\times(P_{N,\bsz_r}^\perp\setminus\{\bszero\})$: equality of two images implies $\bsh-\bsh'\in P_{N,\bsz_r}^\perp$, and the definition of $K_{d,\alpha,\bsgamma}(\bsz_r)$ then gives $\bsh=\bsh'$. Consequently,
\[
\sum_{\bsh\in K_{d,\alpha,\bsgamma}(\bsz_r)}
\sum_{\bsl\in P_{N,\bsz_r}^\perp\setminus\{\bszero\}}
|\widehat f(\bsh+\bsl)|^2r_{2\alpha,\bsgamma}(\bsh+\bsl)
\leq\|f\|_{d,\alpha,\bsgamma}^2,
\]
which justifies the last inequality.
    Because $\bsh+\bsl\notin \Acal_{d,\alpha,\bsgamma}(N_2)$ for $\bsh\in K_{d,\alpha,\bsgamma}(\bsz_r)$ and $\bsl\in P^\perp_{N,\bsz_r}\setminus\{\bszero\}$,
 \begin{align*}
    \sup_{\bsh\in K_{d,\alpha,\bsgamma}(\bsz_r)} \sum_{\bsl\in P^\perp_{N,\bsz_r}\setminus\{\bszero\} }\frac{1}{r_{2\alpha,\bsgamma}(\bsh+\bsl)} & \leq \sup_{\bsh\in \ZZ^d} \sum_{\bsl'\in (\bsh+ P^\perp_{N,\bsz_r})\setminus \Acal_{d,\alpha,\bsgamma}(N_2)}\frac{1}{r_{2\alpha,\bsgamma}(\bsl')} \\
    & \leq \sum_{m=0}^\infty \frac{1}{(2^m N_{2})^{2\alpha} }
    \sup_{\bsh\in\ZZ^d}|(\bsh+ P^\perp_{N,\bsz_r})\cap \Acal_{d,\alpha,\bsgamma}(2^{m+1}N_{2})|.   
    \end{align*}
    
Let $m_0\in\mathbb Z$ be such that $2^{m_0-1}<N_2\leq2^{m_0}$. For $L=2^{m+1}N_2$, the integer in Lemma~\ref{lem:fiber} is $m+m_0+1$. Hence, for $r\in\mathcal R$,
\begin{align*}
&\sup_{\bsh\in K_{d,\alpha,\bsgamma}(\bsz_r)}
\sum_{\bsl\in P_{N,\bsz_r}^\perp\setminus\{\bszero\}}
\frac1{r_{2\alpha,\bsgamma}(\bsh+\bsl)}\\
&\leq\frac{3^d}{N_2^{2\alpha}}
\sum_{m=0}^\infty
\frac{(m+m_0+d+1)^d}{2^{2\alpha m}}
\left(1+4^d\frac{2^{m+m_0+1}}{\rho_{2\alpha,\bsgamma}(\bsz_r)^{1/(2\alpha)}}\right)\\
&\leq\frac{3^d(1+4^{d+1})}{N_2^{2\alpha}}
\sum_{m=0}^\infty
\frac{(m+m_0+d+1)^d}{2^{(2\alpha-1)m}}\\
&\leq C'_{d,\alpha}\frac{(1+\log_2N_2)^d}{N_2^{2\alpha}}.
\end{align*}
In the second inequality we used $\rho_{2\alpha,\bsgamma}(\bsz_r)^{1/(2\alpha)}\geq N_2$ and $2^{m_0}<2N_2$. For the last inequality, note that $m_0\leq1+\log_2N_2$ and
\[
\sum_{m=0}^\infty\frac{(m+m_0+d+1)^d}{2^{(2\alpha-1)m}}
\leq(m_0+d+1)^d
\sum_{m=0}^\infty\frac{(m+1)^d}{2^{(2\alpha-1)m}}
\leq(d+2)^d(1+\log_2N_2)^d
\sum_{m=0}^\infty\frac{(m+1)^d}{2^{(2\alpha-1)m}},
\]
and the final series converges because $\alpha>1/2$.
 Our conclusion follows by putting the above bounds into \eqref{eqn:L2errordecomp}.
\end{proof}

\begin{remark}\label{rmk:RislogN}
 Lemma~\ref{lem:Nstar} ensures that $N_2\geq\sqrt N$ for all $N$ exceeding a finite threshold $N_{\min}$ by applying the lemma, for example, with exponent loss $1/4$ and then increasing the threshold.  In this regime, Lemma~\ref{lem:Adcount} gives the following bound on the failure probabilities $\delta_1$ and $\delta_2$ in Theorem~\ref{thm:infty} and Theorem~\ref{thm:2}:
$$\delta_1\leq \delta_2\leq \frac{N}{2}\left(\frac{8}{1+(\tau/2)\log N}\right)^{\lceil R/2 \rceil}.$$
Hence, for any $\delta_{\max}\in (0,1]$ and $N\geq \max( N_{\min}, \exp(30/\tau))$, the choice 
$$R= 2\lceil\log_2 (N/\delta_{\max})\rceil-1 $$
guarantees $\delta_1\leq \delta_2\leq \delta_{\max}$. For fixed $\delta_{\max}$, this $R$ scales logarithmically with $N$.

Furthermore, under the tractability condition $\sum_{j=1}^\infty \gamma_j^{1/(2\alpha)}<\infty$, Lemma~\ref{lem:Nstar} shows that the threshold $N_{\min}$ is independent of the dimension $d$, implying that the constraint $\delta_1\leq \delta_2\leq 1$ introduces no dimensional dependence.

\end{remark}

\begin{remark}
The median lattice algorithm can be viewed as a variant of multiple lattice-based algorithms, as the generating vectors $\bsz_r$ for which the median aggregation is attained can be different for each frequency $\bsh\in \Acal_{d,\alpha,\bsgamma}(N_2)$. While similar probabilistic strategies employing multiple lattices have been proposed in \cite{CG26,K19, KV19}, the median lattice approach offers a distinct advantage: it eliminates the need for an explicit verification step to detect aliasing terms across all frequencies in $\mathcal{A}_{d,\alpha,\boldsymbol{\gamma}}(N_2)$. Instead, with high probability, for every reconstructed frequency a majority of the coefficient estimators contain no aliasing contribution from another frequency in $\mathcal A_{d,\alpha,\boldsymbol\gamma}(N_2)$, simplifying the implementation and reducing the computational overhead.
\end{remark}

\subsection{$L_p$-approximation}\label{sec_Lp}

We prove Theorem~\ref{thm:main} in this section. Let $E_2$ denote the event in Lemma~\ref{cor:whp2}, and, for $\omega\in E_2$, write $A=A^\omega$. On $E_2$, for every $\bsh\in\Acal_{d,\alpha,\bsgamma}(N_2)$,
\[
|\{r\in\{1{:}R\}:\bsh\in K_{d,\alpha,\bsgamma}(\bsz_r)\}|
\geq|\{r\in\mathcal R:\bsh\in K_{d,\alpha,\bsgamma}(\bsz_r)\}|
\geq\lceil R/2\rceil.
\]
Therefore the proof of Theorem~\ref{thm:infty}, as well as Theorem~\ref{thm:2}, holds on the same event $E_2$, whose probability is at least $1-\delta_2$. Since the $L_p$-approximation error for $1 \le p \le 2$ is bounded from above by the $L_2$-approximation error and we can estimate $(1+\log_2 N_2)^{d/2}$ by $C_{d,\beta} N^{\beta}$ for any $\beta > 0$, this part follows from Theorem~\ref{thm:2}. The bound independent of the dimension for $p=\infty$ follows directly from Theorem~\ref{thm:infty}.

It remains to prove the result for $2 < p < \infty$. We use the interpolation inequality $\|g\|_{L_p} \le \|g\|_{L_2}^{2/p} \|g\|_{L_\infty}^{1-2/p}$, which holds for $g \in L_2 \cap L_\infty$ (see, e.g., \cite[Theorem 2.11]{adams2003sobolev}). Since every function $f$ in the 
weighted periodic Korobov space \( \Hcal_{d,\alpha,\bsgamma}\) with $\alpha>1/2$ is bounded and continuous, we have $f \in L_\infty$. Further, $A(f) \in L_\infty$ and so the norm inequality can be applied. 

By choosing a sufficiently large $C_{d,\beta}>0$, we can estimate $(1+\log_2 N_2)^{d/2} \le C_{d,\beta} N^\beta$ for any $\beta > 0$. From Theorem~\ref{thm:infty} we get the error bound 
$$\|f - A(f) \|_{L_\infty} \lesssim \|f\|_{d,\alpha,\bsgamma}\, R\, N^{-\alpha + 1/2 + \beta},$$
where the notation $\lesssim$ means the inequality holds up to a constant factor independent of $f,N$ and $R$.
From Theorem~\ref{thm:2} we get the error bound 
$$\|f - A(f)\|_{L_2} \lesssim \|f\|_{d,\alpha,\bsgamma}\, \sqrt{R} \, N^{-\alpha + \beta}.$$
The occurrences of $\beta$ in the endpoint estimates represent auxiliary positive losses. Given the prescribed loss in Theorem~\ref{thm:main}, these auxiliary losses and the choice of $q>1/(2\alpha)$ are taken sufficiently small so that their total contribution is at most the prescribed $\beta$.
Applying the norm interpolation inequality, we obtain the bound
\begin{equation*}
\|f- A(f)\|_{L_p} \lesssim
  \|f\|_{d,\alpha,\bsgamma}\, R^{1-\frac{1}{p}}\, N^{-\alpha + \frac12 - \frac1p + \beta},
\end{equation*}
and hence
\begin{equation*}
\mathrm{err}(\Hcal_{d,\alpha,\bsgamma}, L_p, A) \lesssim R^{1-\frac{1}{p}}\, N^{-\alpha + \frac12 - \frac1p + \beta}, \quad 2 \le p \le \infty.
\end{equation*}
Thus, the result follows.

\section*{Acknowledgements}

We, the authors of this paper, have all been greatly inspired by the work of Henryk Wo\'{z}niakowski, and some of us have had the privilege to cooperate with him. We would like to thank Henryk for his numerous contributions to Information-Based Complexity and related fields, for his support, and---what is most important---his friendship. 

The third author acknowledges the support of the Australian Research Council grant DP220101811. The fourth author acknowledges the support of the Japan Society for the Promotion of Science (JSPS) KAKENHI Grant Numbers JP23K03210 and JP26K22266. The fifth author acknowledges the support of the Austrian Science Fund (FWF) Project  P 34808/Grant DOI: 10.55776/P34808. For open access purposes, the authors have applied a CC BY public copyright license to any author accepted manuscript version arising from this submission.

\bibliographystyle{elsarticle-harv} 
\bibliography{ref}

@book {adams2003sobolev,
    AUTHOR = {Adams, Robert A. and Fournier, John J. F.},
     TITLE = {Sobolev Spaces},
    SERIES = {Pure and Applied Mathematics (Amsterdam)},
    VOLUME = {140},
   EDITION = {Second},
 PUBLISHER = {Elsevier/Academic Press, Amsterdam},
      YEAR = {2003},
     PAGES = {xiv+305},
      ISBN = {0-12-044143-8},
}

@article {HN03,
    AUTHOR = {Hickernell, F.~J. and Niederreiter, H.},
     TITLE = {The existence of good extensible rank-1 lattices},
   JOURNAL = {J. Complexity},
  FJOURNAL = {Journal of Complexity},
    VOLUME = {19},
      YEAR = {2003},
    NUMBER = {3},
     PAGES = {286--300},
       DOI = {10.1016/S0885-064X(02)00026-2},
}

@article {NNS22,
    AUTHOR = {Nguyen, Van Dung and Nguyen, Van Kien and Sickel, Winfried},
     TITLE = {{$s$}-numbers of embeddings of weighted {W}iener algebras},
   JOURNAL = {J. Approx. Theory},
  FJOURNAL = {Journal of Approximation Theory},
    VOLUME = {279},
      YEAR = {2022},
     PAGES = {Paper No. 105745, 26},
      ISSN = {0021-9045,1096-0430},
       DOI = {10.1016/j.jat.2022.105745},
}

@article {DKU23,
    AUTHOR = {Dolbeault, Matthieu and Krieg, David and Ullrich, Mario},
     TITLE = {A sharp upper bound for sampling numbers in {$L_2$}},
   JOURNAL = {Appl. Comput. Harmon. Anal.},
  FJOURNAL = {Applied and Computational Harmonic Analysis},
    VOLUME = {63},
      YEAR = {2023},
     PAGES = {113--134},
      ISSN = {1063-5203,1096-603X},
       DOI = {10.1016/j.acha.2022.12.001},
}

@article{CDKU26,
  title={Constructive discretization and approximation in reproducing kernel Hilbert spaces},
  author={Chkifa, Abdellah and Dolbeault, Matthieu and Krieg, David and Ullrich, Mario},
  journal={arXiv preprint arXiv:2602.18719},
  DOI = {10.48550/arXiv.2602.18719},
  year={2026}
}

@article {T93,
    AUTHOR = {Temlyakov, V. N.},
     TITLE = {On approximate recovery of functions with bounded mixed
              derivative},
   JOURNAL = {J. Complexity},
  FJOURNAL = {Journal of Complexity},
    VOLUME = {9},
      YEAR = {1993},
    NUMBER = {1},
     PAGES = {41--59},
      ISSN = {0885-064X,1090-2708},
       DOI = {10.1006/jcom.1993.1004},
}

@article {BDSU16,
    AUTHOR = {Byrenheid, Glenn and D\~ung, Dinh and Sickel, Winfried and
              Ullrich, Tino},
     TITLE = {Sampling on energy-norm based sparse grids for the optimal
              recovery of {S}obolev type functions in {$H^\gamma$}},
   JOURNAL = {J. Approx. Theory},
  FJOURNAL = {Journal of Approximation Theory},
    VOLUME = {207},
      YEAR = {2016},
     PAGES = {207--231},
      ISSN = {0021-9045,1096-0430},
       DOI = {10.1016/j.jat.2016.02.012},
}

@article{PKG25,
  title={${L_2}$-approximation using median lattice algorithms},
  author={Pan, Zexin and Kritzer, Peter and Goda, Takashi},
  journal={Constructive Approximation},
  year={2026},
  DOI={10.1007/s00365-026-09763-6}
}

@article {BGKS24,
  title={Minimal subsampled rank-1 lattices for multivariate approximation with optimal convergence rate},
  author={Bartel, Felix and Gilbert, Alexander D. and Kuo, Frances Y. and Sloan, Ian H.},
   JOURNAL = {Math. Comp.},
  FJOURNAL = {Mathematics of Computation},
      YEAR = {Published Online, 2026},
       DOI = {10.1090/mcom/4212}
}

@article {BKTV17,
    AUTHOR = {Byrenheid, Glenn and K\"{a}mmerer, Lutz and Ullrich, Tino and
              Volkmer, Toni},
     TITLE = {Tight error bounds for rank-1 lattice sampling in spaces of
              hybrid mixed smoothness},
   JOURNAL = {Numer. Math.},
  FJOURNAL = {Numerische Mathematik},
    VOLUME = {136},
      YEAR = {2017},
    NUMBER = {4},
     PAGES = {993--1034},
      ISSN = {0029-599X,0945-3245},
       DOI = {10.1007/s00211-016-0861-7},
}

@article{CDG25, 
    title={A lattice algorithm with multiple shifts for function approximation in {K}orobov spaces},
    author={Cai, M. and Dick, J. and Goda, T.},
    journal={arXiv:2511.09071},
    DOI = {10.48550/arXiv.2511.09071},
    year={2025}
}

@article{CG26,
    title={A note on approximation in weighted {K}orobov spaces via multiple rank-1 lattices}, 
    author={Cai, Mou and Goda, Takashi},
   JOURNAL = {J. Complexity},
  FJOURNAL = {Journal of Complexity},
    volume={96},
    pages={102054},
    year={2026},
    DOI={10.1016/j.jco.2026.102054}
}

@incollection {CKNS20,
    AUTHOR = {Cools, Ronald and Kuo, Frances Y. and Nuyens, Dirk and Sloan,
              Ian H.},
     TITLE = {Lattice algorithms for multivariate approximation in periodic
              spaces with general weight parameters},
 BOOKTITLE = {75 years of {M}athematics of {C}omputation},
    SERIES = {Contemp. Math.},
    VOLUME = {754},
     PAGES = {93--113},
 PUBLISHER = {Amer. Math. Soc., RI},
      YEAR = {2020},
      ISBN = {978-1-4704-5163-9},
       DOI = {10.1090/conm/754/15150},
}

@article {CKNS21,
    AUTHOR = {Cools, Ronald and Kuo, Frances Y. and Nuyens, Dirk and Sloan,
              Ian H.},
     TITLE = {Fast component-by-component construction of lattice algorithms
              for multivariate approximation with {POD} and {SPOD} weights},
   JOURNAL = {Math. Comp.},
  FJOURNAL = {Mathematics of Computation},
    VOLUME = {90},
      YEAR = {2021},
    NUMBER = {328},
     PAGES = {787--812},
      ISSN = {0025-5718,1088-6842},
       DOI = {10.1090/mcom/3586},
}

@article {D04,
    AUTHOR = {Dick, Josef},
     TITLE = {On the convergence rate of the component-by-component
              construction of good lattice rules},
   JOURNAL = {J. Complexity},
  FJOURNAL = {Journal of Complexity},
    VOLUME = {20},
      YEAR = {2004},
    NUMBER = {4},
     PAGES = {493--522},
      ISSN = {0885-064X,1090-2708},
       DOI = {10.1016/j.jco.2003.11.008},
}

@article {DG21,
    AUTHOR = {Dick, Josef and Goda, Takashi},
     TITLE = {Stability of lattice rules and polynomial lattice rules
              constructed by the component-by-component algorithm},
   JOURNAL = {J. Comput. Appl. Math.},
  FJOURNAL = {Journal of Computational and Applied Mathematics},
    VOLUME = {382},
      YEAR = {2021},
     PAGES = {Paper No. 113062, 16},
      ISSN = {0377-0427,1879-1778},
       DOI = {10.1016/j.cam.2020.113062},
}

@article {DGS22,
    AUTHOR = {Dick, Josef and Goda, Takashi and Suzuki, Kosuke},
     TITLE = {Component-by-component construction of randomized rank-1
              lattice rules achieving almost the optimal randomized error
              rate},
   JOURNAL = {Math. Comp.},
  FJOURNAL = {Mathematics of Computation},
    VOLUME = {91},
      YEAR = {2022},
    NUMBER = {338},
     PAGES = {2771--2801},
      ISSN = {0025-5718,1088-6842},
       DOI = {10.1090/mcom/3769},
}

@book {DKP22,
    AUTHOR = {Dick, Josef and Kritzer, Peter and Pillichshammer, Friedrich},
     TITLE = {Lattice {R}ules---{N}umerical {I}ntegration, {A}pproximation, and
              {D}iscrepancy},
    SERIES = {Springer Series in Computational Mathematics},
    VOLUME = {58},
 PUBLISHER = {Springer, Cham},
      YEAR = {2022},
     PAGES = {xvi+580},
      ISBN = {978-3-031-09950-2; 978-3-031-09951-9},
       DOI = {10.1007/978-3-031-09951-9},
}

@article {DKS13,
    AUTHOR = {Dick, Josef and Kuo, Frances Y. and Sloan, Ian H.},
     TITLE = {High-dimensional integration: the quasi-{M}onte {C}arlo way},
   JOURNAL = {Acta Numer.},
  FJOURNAL = {Acta Numerica},
    VOLUME = {22},
      YEAR = {2013},
     PAGES = {133--288},
      ISSN = {0962-4929,1474-0508},
       DOI = {10.1017/S0962492913000044},
}

@article {DSWW06,
    AUTHOR = {Dick, Josef and Sloan, Ian H. and Wang, Xiaoqun and
              Wo\'zniakowski, Henryk},
     TITLE = {Good lattice rules in weighted {K}orobov spaces with general
              weights},
   JOURNAL = {Numer. Math.},
  FJOURNAL = {Numerische Mathematik},
    VOLUME = {103},
      YEAR = {2006},
    NUMBER = {1},
     PAGES = {63--97},
      ISSN = {0029-599X,0945-3245},
       DOI = {10.1007/s00211-005-0674-6},
}

@article {G26,
    AUTHOR = {Goda, Takashi},
     TITLE = {A randomized lattice rule without component-by-component
              construction},
   JOURNAL = {Math. Comp.},
  FJOURNAL = {Mathematics of Computation},
    VOLUME = {95},
      YEAR = {2026},
    NUMBER = {357},
     PAGES = {339--361},
      ISSN = {0025-5718,1088-6842},
       DOI = {10.1090/mcom/4051},
}

@article {GK26,
    AUTHOR = {Goda, Takashi and Krieg, David},
     TITLE = {A simple universal algorithm for high-dimensional integration},
   JOURNAL = {Numer. Math.},
  FJOURNAL = {Numerische Mathematik},
    VOLUME = {158},
      YEAR = {2026},
    NUMBER = {1},
     PAGES = {229--248},
      ISSN = {0029-599X,0945-3245},
       DOI = {10.1007/s00211-025-01517-z},
}

@article {GL22,
    AUTHOR = {Goda, Takashi and L'Ecuyer, Pierre},
     TITLE = {Construction-free median quasi--{M}onte {C}arlo rules for
              function spaces with unspecified smoothness and general
              weights},
   JOURNAL = {SIAM J. Sci. Comput.},
  FJOURNAL = {SIAM Journal on Scientific Computing},
    VOLUME = {44},
      YEAR = {2022},
    NUMBER = {4},
     PAGES = {A2765--A2788},
      ISSN = {1064-8275,1095-7197},
       DOI = {10.1137/22M1473625},
}

@article {GSM24,
    AUTHOR = {Goda, Takashi and Suzuki, Kosuke and Matsumoto, Makoto},
     TITLE = {A universal median quasi--{M}onte {C}arlo integration},
   JOURNAL = {SIAM J. Numer. Anal.},
  FJOURNAL = {SIAM Journal on Numerical Analysis},
    VOLUME = {62},
      YEAR = {2024},
    NUMBER = {1},
     PAGES = {533--566},
      ISSN = {0036-1429,1095-7170},
       DOI = {10.1137/22M1525077},
}

@book {HW63,
    AUTHOR = {Hua, Loo Keng and Wang, Yuan},
     TITLE = {Applications of {N}umber {T}heory to {N}umerical {A}nalysis},
      NOTE = {Translated from the Chinese},
 PUBLISHER = {Springer-Verlag, Berlin-New York; Kexue Chubanshe (Science
              Press), Beijing},
      YEAR = {1981},
     PAGES = {ix+241},
      ISBN = {3-540-10382-1},
}

@article {K19,
    AUTHOR = {K\"{a}mmerer, Lutz},
     TITLE = {Constructing spatial discretizations for sparse multivariate
              trigonometric polynomials that allow for a fast discrete
              {F}ourier transform},
   JOURNAL = {Appl. Comput. Harmon. Anal.},
  FJOURNAL = {Applied and Computational Harmonic Analysis},
    VOLUME = {47},
      YEAR = {2019},
    NUMBER = {3},
     PAGES = {702--729},
      ISSN = {1063-5203,1096-603X},
       DOI = {10.1016/j.acha.2017.11.008},
}

@article {KPV15,
    AUTHOR = {K\"{a}mmerer, Lutz and Potts, Daniel and Volkmer, Toni},
     TITLE = {Approximation of multivariate periodic functions by
              trigonometric polynomials based on sampling along rank-1
              lattice with generating vector of {K}orobov form},
   JOURNAL = {J. Complexity},
  FJOURNAL = {Journal of Complexity},
    VOLUME = {31},
      YEAR = {2015},
    NUMBER = {3},
     PAGES = {424--456},
      ISSN = {0885-064X,1090-2708},
       DOI = {10.1016/j.jco.2014.09.001},
}

@article {KV19,
    AUTHOR = {K\"{a}mmerer, Lutz and Volkmer, Toni},
     TITLE = {Approximation of multivariate periodic functions based on
              sampling along multiple rank-1 lattices},
   JOURNAL = {J. Approx. Theory},
  FJOURNAL = {Journal of Approximation Theory},
    VOLUME = {246},
      YEAR = {2019},
     PAGES = {1--27},
      ISSN = {0021-9045,1096-0430},
       DOI = {10.1016/j.jat.2019.05.001},
}

@article{KPUU25, 
    AUTHOR = {Krieg, David and Pozharska, Kateryna and Ullrich, Mario and Ullrich, Tino},
     TITLE = {Sampling recovery in {$L_2$} and other norms},
   JOURNAL = {Math. Comp.},
  FJOURNAL = {Mathematics of Computation},
      YEAR = {Published Online, 2025},
       DOI = {10.1090/mcom/4148}
}

@article {KKNU19,
    AUTHOR = {Kritzer, Peter and Kuo, Frances Y. and Nuyens, Dirk and
              Ullrich, Mario},
     TITLE = {Lattice rules with random {$n$} achieve nearly the optimal
              {$\mathcal{O}(n^{-\alpha-1/2})$} error independently of the
              dimension},
   JOURNAL = {J. Approx. Theory},
  FJOURNAL = {Journal of Approximation Theory},
    VOLUME = {240},
      YEAR = {2019},
     PAGES = {96--113},
      ISSN = {0021-9045,1096-0430},
       DOI = {10.1016/j.jat.2018.09.011},
}

@article {KNR19,
    AUTHOR = {Kunsch, Robert J. and Novak, Erich and Rudolf, Daniel},
     TITLE = {Solvable integration problems and optimal sample size
              selection},
   JOURNAL = {J. Complexity},
  FJOURNAL = {Journal of Complexity},
    VOLUME = {53},
      YEAR = {2019},
     PAGES = {40--67},
      ISSN = {0885-064X,1090-2708},
       DOI = {10.1016/j.jco.2018.10.007},
}

@article {K03,
    AUTHOR = {Kuo, F. Y.},
     TITLE = {Component-by-component constructions achieve the optimal rate
              of convergence for multivariate integration in weighted
              {K}orobov and {S}obolev spaces},
   JOURNAL = {J. Complexity},
  FJOURNAL = {Journal of Complexity},
    VOLUME = {19},
      YEAR = {2003},
    NUMBER = {3},
     PAGES = {301--320},
      ISSN = {0885-064X,1090-2708},
       DOI = {10.1016/S0885-064X(03)00006-2},
}

@article {KMNN21,
    AUTHOR = {Kuo, Frances Y. and Migliorati, Giovanni and Nobile, Fabio and
              Nuyens, Dirk},
     TITLE = {Function integration, reconstruction and approximation using
              rank-1 lattices},
   JOURNAL = {Math. Comp.},
  FJOURNAL = {Mathematics of Computation},
    VOLUME = {90},
      YEAR = {2021},
    NUMBER = {330},
     PAGES = {1861--1897},
      ISSN = {0025-5718,1088-6842},
       DOI = {10.1090/mcom/3595},
}

@article {KNW23,
    AUTHOR = {Kuo, Frances Y. and Nuyens, Dirk and Wilkes, Laurence},
     TITLE = {Random-prime-fixed-vector randomised lattice-based algorithm
              for high-dimensional integration},
   JOURNAL = {J. Complexity},
  FJOURNAL = {Journal of Complexity},
    VOLUME = {79},
      YEAR = {2023},
     PAGES = {Paper No. 101785, 28},
      ISSN = {0885-064X,1090-2708},
       DOI = {10.1016/j.jco.2023.101785},
}

@incollection {KSW06,
    AUTHOR = {Kuo, Frances Y. and Sloan, Ian H. and Wo\'zniakowski, Henryk},
     TITLE = {Lattice rules for multivariate approximation in the worst case
              setting},
 BOOKTITLE = {Monte {C}arlo and quasi-{M}onte {C}arlo methods 2004},
     PAGES = {289--330},
 PUBLISHER = {Springer, Berlin},
      YEAR = {2006},
      ISBN = {978-3-540-25541-3; 3-540-25541-9},
       DOI = {10.1007/3-540-31186-6\_18},
}

@incollection {LH03,
    AUTHOR = {Li, Dong and Hickernell, Fred J.},
     TITLE = {Trigonometric spectral collocation methods on lattices},
 BOOKTITLE = {Recent {A}dvances in {S}cientific {C}omputing and {P}artial
              {D}ifferential {E}quations ({H}ong {K}ong, 2002)},
    SERIES = {Contemp. Math.},
    VOLUME = {330},
     PAGES = {121--132},
 PUBLISHER = {Amer. Math. Soc., Providence, RI},
      YEAR = {2003},
      ISBN = {0-8218-3155-0},
       DOI = {10.1090/conm/330/05887},
}

@book {N92,
    AUTHOR = {Niederreiter, Harald},
     TITLE = {Random {N}umber {G}eneration and {Q}uasi-{M}onte {C}arlo {M}ethods},
    SERIES = {CBMS-NSF Regional Conference Series in Applied Mathematics},
    VOLUME = {63},
 PUBLISHER = {Society for Industrial and Applied Mathematics (SIAM),
              Philadelphia, PA},
      YEAR = {1992},
     PAGES = {vi+241},
      ISBN = {0-89871-295-5},
       DOI = {10.1137/1.9781611970081},
}

@article {NP09,
    AUTHOR = {Niemiro, Wojciech and Pokarowski, Piotr},
     TITLE = {Fixed precision {MCMC} estimation by median of products of
              averages},
   JOURNAL = {J. Appl. Probab.},
  FJOURNAL = {Journal of Applied Probability},
    VOLUME = {46},
      YEAR = {2009},
    NUMBER = {2},
     PAGES = {309--329},
      ISSN = {0021-9002,1475-6072},
       DOI = {10.1239/jap/1245676089},
}

@book {NW08,
    AUTHOR = {Novak, Erich and Wo\'zniakowski, Henryk},
     TITLE = {Tractability of {M}ultivariate {P}roblems. {V}ol. 1: {L}inear
              {I}nformation},
    SERIES = {EMS Tracts in Mathematics},
    VOLUME = {6},
 PUBLISHER = {European Mathematical Society (EMS), Z\"urich},
      YEAR = {2008},
     PAGES = {xii+384},
      ISBN = {978-3-03719-026-5},
       DOI = {10.4171/026},
}

@article {NC06,
    AUTHOR = {Nuyens, Dirk and Cools, Ronald},
     TITLE = {Fast algorithms for component-by-component construction of
              rank-1 lattice rules in shift-invariant reproducing kernel
              {H}ilbert spaces},
   JOURNAL = {Math. Comp.},
  FJOURNAL = {Mathematics of Computation},
    VOLUME = {75},
      YEAR = {2006},
    NUMBER = {254},
     PAGES = {903--920},
      ISSN = {0025-5718,1088-6842},
       DOI = {10.1090/S0025-5718-06-01785-6},
}

@article {P26,
    AUTHOR = {Pan, Zexin},
     TITLE = {Automatic optimal-rate convergence of randomized nets using
              median-of-means},
   JOURNAL = {Math. Comp.},
  FJOURNAL = {Mathematics of Computation},
    VOLUME = {95},
      YEAR = {2026},
    NUMBER = {359},
     PAGES = {1415--1446},
      ISSN = {0025-5718,1088-6842},
       DOI = {10.1090/mcom/4093},
}

@article{PGK25,
  title={Universal {$L_2$}-approximation using median lattice algorithms},
  author={Pan, Zexin and Goda, Takashi and Kritzer, Peter},
  journal={arXiv preprint arXiv:2509.24582},
  DOI = {10.48550/arXiv.2509.24582},
  year={2025},
}

@article {PO23,
    AUTHOR = {Pan, Zexin and Owen, Art B.},
     TITLE = {Super-polynomial accuracy of one dimensional randomized nets
              using the median of means},
   JOURNAL = {Math. Comp.},
  FJOURNAL = {Mathematics of Computation},
    VOLUME = {92},
      YEAR = {2023},
    NUMBER = {340},
     PAGES = {805--837},
      ISSN = {0025-5718,1088-6842},
       DOI = {10.1090/mcom/3791},
}

@article {PO24,
    AUTHOR = {Pan, Zexin and Owen, Art B.},
     TITLE = {Super-polynomial accuracy of multidimensional randomized nets
              using the median-of-means},
   JOURNAL = {Math. Comp.},
  FJOURNAL = {Mathematics of Computation},
    VOLUME = {93},
      YEAR = {2024},
    NUMBER = {349},
     PAGES = {2265--2289},
      ISSN = {0025-5718,1088-6842},
       DOI = {10.1090/mcom/3880},
}

@book {SJ94,
    AUTHOR = {Sloan, I. H. and Joe, S.},
     TITLE = {Lattice methods for multiple integration},
    SERIES = {Oxford Science Publications},
 PUBLISHER = {The Clarendon Press, Oxford University Press, New York},
      YEAR = {1994},
     PAGES = {xii+239},
      ISBN = {0-19-853472-8},
}

@article {SR02,
    AUTHOR = {Sloan, I. H. and Reztsov, A. V.},
     TITLE = {Component-by-component construction of good lattice rules},
   JOURNAL = {Math. Comp.},
  FJOURNAL = {Mathematics of Computation},
    VOLUME = {71},
      YEAR = {2002},
    NUMBER = {237},
     PAGES = {263--273},
      ISSN = {0025-5718,1088-6842},
       DOI = {10.1090/S0025-5718-01-01342-4},
}

@article {SW98,
    AUTHOR = {Sloan, Ian H. and Wo\'zniakowski, Henryk},
     TITLE = {When are quasi-{M}onte {C}arlo algorithms efficient for
              high-dimensional integrals?},
   JOURNAL = {J. Complexity},
  FJOURNAL = {Journal of Complexity},
    VOLUME = {14},
      YEAR = {1998},
    NUMBER = {1},
     PAGES = {1--33},
      ISSN = {0885-064X,1090-2708},
       DOI = {10.1006/jcom.1997.0463},
}



\end{document}